\documentclass[12pt]{amsart}
\usepackage[utf8]{inputenc}
\usepackage{amssymb}
\usepackage{graphicx}
\usepackage{graphicx,color}
\usepackage{latexsym}
\usepackage[all]{xy}
\usepackage{mathrsfs}
\usepackage{enumerate}
\usepackage{amssymb}
\usepackage{tikz}
\usepackage{breqn}

\makeatletter
\renewcommand{\subsection}{\@startsection
{subsection}{2}{0mm}{\baselineskip}{-0.25cm}
{\normalfont\normalsize\em}}
\makeatother

\def\negei{\mathbf e_i}
\def\negej{\mathbf e_j}

\def\negP{\mathbf P}
\def\negQ{\mathbf Q}
\def\negalpha{\text{\boldmath$\alpha$}}

\def\neg1{\text{\boldmath$1$}}
\def\negbeta{\text{\boldmath$\beta$}}
\def\neggamma{\text{\boldmath$\gamma$}}
\def\negeta{\text{\boldmath$\eta$}}

\def\neggamma{\text{\boldmath$\gamma$}}

\def\negeta{\text{\boldmath$\eta$}}

\def\neg1{\text{\boldmath$1$}}

\def\hH{\widehat{H}}

\def\hGamma{\widehat{\Gamma}}
\def\hLambda{\widehat{\Lambda}}
\def\cC{\mathcal C}

\def\cL{\mathcal L}

\def\cX{\mathcal X}

\def\NN{\mathbb{N}}

\def\ZZ{\mathbb{Z}}

\def\FF{\mathbb{F}}

\def\Fq{{\mathbb{F}_q}}

\DeclareMathOperator{\divv}{div}
\DeclareMathOperator{\divvp}{div_\infty}

\DeclareMathOperator{\lub}{lub}

\newtheorem{theorem}{Theorem}[section]
\newtheorem{proposition}[theorem]{Proposition}
\newtheorem{corollary}[theorem]{Corollary}
\newtheorem{lemma}[theorem]{Lemma}

{\theoremstyle{definition}
\newtheorem{definition}[theorem]{Definition}
\newtheorem{example}[theorem]{Example}

}

{\theoremstyle{remark}
\newtheorem{remark}[theorem]{Remark}
}

\title[On Weierstrass gaps at several points]{On Weierstrass gaps at several points}

\author[W. Ten\'orio \and G. Tizziotti]{Wanderson Ten\'orio \and Guilherme Tizziotti}
 \address{Universidade Federal de Uberl\^andia (UFU), Faculdade de Matem\' atica, Av.~J.~N. \' Avila~2121, 38408-902, Uberl\^andia, MG, Brazil}
  \email{dersonwt@yahoo.com.br, guilhermect@ufu.br}
	  \date{\today}

\begin{document}


\maketitle

\begin{abstract} 
We consider the problem of determining Weierstrass gaps and pure Weierstrass gaps at several points. Using the notion of relative maximality in generalized Weierstrass semigroups due to Delgado \cite{D}, we present a description of these elements which generalizes the approach of Homma and Kim \cite{HK} given for pairs. Through this description, we study the gaps and pure gaps at several points on a certain family of curves with separated variables.
\end{abstract}

\section{Introduction}

Let $\mathcal{X}$ be a (nonsingular, projective, geometrically irreducible, algebraic) curve of positive genus $g(\cX)$ defined over $\mathbb{F}_q$, the finite field with $q$ elements. Let $\mathbb{F}_q(\mathcal{X})$ be its associated function field and let $\mbox{Div}(\cX)$ be the set of divisors on $\cX$. For $G\in \mbox{Div}(\cX)$, as usual, $\cL(G)$ will denote the Riemann-Roch space $\{h\in \FF_q(\cX) \ : \ \divv(h)+G\geq 0\}\cup \{0\}$ and $\ell(G)$ its dimension as an $\FF_q$-vector space. Given $m\geq 1$ pairwise distinct rational points $Q_1,\ldots,Q_m$ on $\mathcal{X}$, consider the $m$-tuple $\mathbf{Q} = (Q_{1}, \ldots , Q_{m})$. The \textit{(classical) Weierstrass semigroup} of $\cX$ at $\negQ$, introduced in Arbarello et al. \cite{Arb}, is defined as
$$
H(\mathbf{Q}) := \left\{(\alpha_{1}, \ldots, \alpha_{m}) \in \mathbb{N}_{0}^ {m} \mbox{ : } \exists \ h \in \mathbb{F}_q(\mathcal{X})^\times \mbox{ with } \divvp(h) = \sum_{i=1}^ {m} \alpha_{i}Q_{i} \right\},
$$
where $\divv(h)$ and $\divvp(h)$ stand respectively for the divisor and the divisor of poles of the rational function $h$. The set $H(\mathbf{Q})$ is a subsemigroup of $\NN_0^m$ with respect to addition. Roughly speaking, this semigroup carries information about the Riemann-Roch spaces of divisors of type $D_\negalpha:=\alpha_1Q_1+\cdots+\alpha_mQ_m$, where $\negalpha=(\alpha_1,\ldots,\alpha_m)\in \NN_0^m$; indeed, under the assumption $q\geq m$, we have $\negalpha\in H(\negQ)$ if and only if ${\ell(D_\negalpha)=\ell(D_\negalpha-Q_j)+1}$ for ${j=1,\ldots,m}$. On the other hand, the elements in the finite complement ${G(\mathbf{Q}):=\NN^m_0\backslash H(\mathbf{Q})}$, called \emph{Weierstrass gaps} of $\cX$ at $\mathbf{Q}$ (or simply \emph{gaps}), also play an important role in the study of such Riemann-Roch spaces. A \emph{pure Weierstrass gap} (or simply \textit{pure gap}) is an $m$-tuple $\negalpha \in G(\mathbf{Q})$ such that $\ell(D_\negalpha) = \ell(D_\negalpha - Q_j)$ for  $j=1,\ldots,m$. The set of pure gaps of $\cX$ at $\mathbf{Q}$ will be denoted by $G_0(\mathbf{Q})$; for further information concerning these objects, we refer the reader to the survey \cite{CK}.

At one single point, the cardinality of the gaps is precisely the genus of $\mathcal{X}$, and its study is a classical research theme in algebraic-geometry. For the several points case, some questions concerning $G(\negQ)$ and $G_0(\negQ)$ remain unanswered. One of central interests in these objects comes from their applications to the analysis of multipoint algebraic-geometric codes. In particular, the use of pure gaps, introduced first by Homma and Kim in \cite{HK} for pairs and extended to several points by Carvalho and Torres \cite{CT}, is related to improvements on the Goppa bound of the minimum distance of these codes. For $m=2$, the sets $G(Q_1,Q_2)$ and $G_0(Q_1,Q_2)$ were extensively studied by Homma and Kim in a serie of papers \cite{H,K,HK}. They introduced a special finite subset of the classical Weierstrass semigroup at a pair carrying information about the whole semigroup, as well as information about the gaps and pure gaps. This set was called later by Matthews \cite{Ma}  the \emph{minimal generating set} of the classical Weierstrass semigroups at pairs, where she extended such notion to the general case $m\geq 2$. Despite such generating sets at several points describe well the elements of their respective semigroups, it seems that they are not directly related with the gaps and pure gaps as in the case $m=2$. 

In this work we study special elements in $\NN^m$ that bring information on gaps and pure gaps. In Theorems \ref{gaps} and \ref{puregaps}, we state a general characterization of gaps and pure gaps at several points which generalizes that one given by Homma and Kim for pairs and has not been noticed and explored previously. To accomplish this, we use the notion of maximality (see Def. \ref{def:maximals}) due to Delgado \cite{D} to generalized Weierstrass semigroups. Delgado's work deals with a generalization of classical Weierstrass semigroups as follows: considering the $m$-tuple $\mathbf{Q}=(Q_1,\ldots,Q_m)$ as before, $R_{\mathbf{Q}}$ denotes the subset of $\mathbb{F}_q(\cX)$ consisting of rational functions which are regular outside $\{Q_1,\ldots,Q_m\}$. The \emph{generalized Weierstrass semigroup} of $\cX$ at $\mathbf{Q}$ is then the set
$$ \widehat{H} (\mathbf{Q}):=\{(-v_{Q_1}(h),\ldots,-v_{Q_m}(h))\in \mathbb{Z}^m \ : \ h\in R_{\mathbf{Q}}\backslash\{0\} \},$$
where $v_{Q_i}$ is the discrete valuation of $\mathbb{F}_q(\mathcal{X})$ associated with $Q_i$. Beelen and Tutas \cite{BT} studied these objects for curves defined over finite fields---in \cite{D} was assumed curves over algebraically closed fields---and present many interesting properties, especially for a pair of points. These structures were recently explored in \cite{MTT}, where there is established connections between the concepts of maximal elements in \cite{D} and generating sets for $\hH(\negQ)$ in the sense of \cite{Ma}. Explicit computations relying on the approach of \cite{MTT} and its outcomes are presented in \cite{TT} for certain special curves with separable variables, denoted by $\cX_{f,g}$ (see Sec. 4).

Many works have appeared recently dedicated mainly to provide descriptions of pure gaps at several points for specific curves. For instance, Bartoli, Quoos, and Zini \cite{bartoli} gave a criterion to find pure gaps at many places and presented families of pure gaps at several points on curves associated with Kummer extensions. In \cite{HY}, Hu and Yang yielded a characterization of gaps and pure gaps at several points also on Kummer extensions. The same authors in \cite{HY2} also furnished an arithmetic/combinatorial description of pure Weierstrass gaps at many totally ramified places on a quotient of the Hermitian curve. To illustrate our approach, we present in this work a study of gaps and pure gaps at several points on $\cX_{f,g}$. This family of curves contains the Kummer extensions and the quotient of the Hermitian curve, studied respectively in \cite{bartoli}, \cite{HY}, and \cite{HY2}.

This paper is organized as follows. In Section 2 we fix the notation and recall some
facts about generalized Weierstrass semigroups at several points and their maximal elements. We also present some properties of relative maximal elements in generalized Weierstrass semigroups. As stated above, Section 3 is devoted to examining the relation of relative maximal elements in the description of gaps and pure gaps at several points. Section 4 specializes the computation of relative maximal elements at several points on the curves $\cX_{f,g}$. This enables us describing the gaps and pure gaps in such curves.

\section{Generalized Weierstrass semigroups and their maximal elements}

\subsection{Settings, notation and preliminary results} In the following, $\mathcal{X}$ will denote a curve defined over $\FF_q$ as in Introduction. For $m\geq 2$, we let $\negQ=(Q_1,\ldots,Q_m)$ be an $m$-tuple of distinct rational points on $\cX$.

Set $I:=\{1,\ldots,m\}$. Following the notation of \cite{D}, we shall consider subsets of $\ZZ^m$ with interesting properties. For $i\in I$, a nonempty $J\subsetneq I$ and $\negalpha=(\alpha_1,\ldots,\alpha_m)\in \ZZ^m$, we shall denote
\begin{itemize}
\item [$\triangleright$] $\nabla_i^m(\negalpha):=\{\negbeta \in \hH(\negQ) \ : \ \beta_i=\alpha_i \ \mbox{and } \beta_j\leq\alpha_j \ \mbox{for } j\neq i\}$;
\item [$\triangleright$] $\overline{\nabla}_J(\negalpha):=\{\negbeta \in \ZZ^m \ : \ \beta_j=\alpha_j \ \mbox{for } \ j\in J \ \mbox{and, } \beta_i<\alpha_i \ \mbox{for } i\not\in J\}$;
\item [$\triangleright$] $\nabla_J(\negalpha):=\overline{\nabla}_J(\negalpha)\cap \hH(\negQ)$;
\item [$\triangleright$] $\overline{\nabla}(\negalpha):=\bigcup_{i=1}^m \overline{\nabla}_i(\negalpha)$, where $\overline{\nabla}_i(\negalpha)=\overline{\nabla}_{\{i\}}(\negalpha)$;
\item [$\triangleright$] $\nabla(\negalpha):=\overline{\nabla}(\negalpha)\cap \hH(\negQ)$;
\item [$\triangleright$] $\textbf{1}_J$ denotes the $m$-tuple whose the $j$-th coordinate $1$ if $j \in J$ and $0$ otherwise; for instance  $\textbf{e}_i = \textbf{1}_{\{i\}} =  (0,\ldots,0,1,0,\ldots,0)$;
\item [$\triangleright$] $\textbf{1}:=(1,\ldots,1)$ and $\textbf{0}:=(0,\ldots,0)$.
\end{itemize}

For $\neggamma=(\gamma_1,\ldots,\gamma_m)\in \mathbb{Z}^m$,  $D_\neggamma$ will denote the divisor $ \gamma_1Q_1+\cdots+\gamma_mQ_m$ on $\cX$.

\begin{proposition}[\cite{MTT}, Proposition 2.1] \label{prop:elements}
Let $\negalpha \in \mathbb{Z}^m$ and assume that $q \geq m$. Then
\begin{enumerate}[\rm (1)]
\item $\negalpha \in \widehat{H} (\mathbf{Q})$ if and only if $\ell(D_{\negalpha}) = \ell(D_{\negalpha} - Q_i)+1$, for all $i \in I$;
\item for $i\in I$, we have $\nabla_{i}^{m}(\negalpha) = \emptyset$ if and only if $\ell(D_{\negalpha})=\ell(D_{\negalpha} - Q_i)$.
\end{enumerate}
\end{proposition}

\begin{remark}\label{rmk:2g}
Let $\negbeta=(\beta_1, \ldots, \beta_m)\in \ZZ^m$. Notice that $\negbeta\in \hH(\negQ)$ if $\beta_1+\cdots+\beta_m\geq 2g(\cX)$.
\end{remark}

\begin{lemma}[\cite{D}, p. 629] \label{teclem} Assume that $q\geq m$. If $\negbeta,\negbeta'\in \hH(\negQ)$ with $\negbeta\neq \negbeta'$ and $\beta_i=\beta'_i$ for some $i\in I$, then there exists $\neggamma\in \hH(\negQ)$ such that $\gamma_i<\beta_i=\beta'_i$ and $\gamma_j\leq \max\{\beta_k,\beta'_k\}$ for $k\neq i$ (and the equality holds if $\beta_k\neq \beta'_k$).
\end{lemma}

We now recall the notion of maximal elements in $\hH(\negQ)$ introduced by Delgado in \cite{D}. As we will see in what follows, these elements will play an important role in description of elements of $\hH(\negQ)$, as well as $G(\negQ)$.

\begin{definition}\label{def:maximals} An element $\negalpha\in \widehat{H} (\mathbf{Q})$ is called \emph{maximal} if ${\nabla(\negalpha)=\emptyset}$. If ${\nabla_J(\negalpha)=\emptyset}$ for every $J\subsetneq I$ with $\#J\geq 2$, we say that $\negalpha$ is \emph{absolute maximal}. If otherwise ${\nabla_J(\negalpha)\neq\emptyset}$ for every $J\subsetneq I$ with $\#J\geq 2$, we say that $\negalpha\in \hH(\negQ)$ is \emph{relative maximal}.
\end{definition}

Observe that the notions of absolute and relative maximality coincide when $m=2$. The sets of absolute and relative maximal elements in $\widehat{H} (\mathbf{Q})$ will be denoted, respectively, by $\hGamma(\mathbf{Q})$ and $\hLambda(\negQ)$. Despite these sets are infinity, they can be finitely determined as follows.\\

For $i=2,\ldots,m$, let $a_i$ be the smallest positive integer $t$ such that ${tQ_{i}-tQ_{i-1}}$ is a principal divisor on $\cX$. We can thus define the region
$$
\cC(\negQ):=\{\negalpha=(\alpha_1, \ldots , \alpha_m)\in \ZZ^m \ : \ 0\leq \alpha_i< a_i \ \mbox{for } i=2,\ldots,m\}.
$$
Let $\negeta^i = (\eta_1^i, \ldots, \eta_m^i)\in \ZZ^m$ be the $m$-tuple whose $j$-th coordinate is
$$
\eta^i_j=\left \lbrace \begin{array}{rcl}
-a_i & , & \mbox{if } j=i-1\\
a_i & , & \mbox{if } j=i\\
0 & , & \mbox{otherwise.}\\
\end{array}\right.
$$
Defining ${\Theta(\negQ):=\{ b_1 \negeta^1+ \ldots + b_{m-1} \negeta^{m-1}\in \ZZ^m \ : b_i \in \ZZ \ \mbox{for } i=1,\ldots,m-1\}}$, we have the following result concerning the finite determination of $\hGamma(\negQ)$ and $\hLambda(\negQ)$.

\begin{theorem}[\cite{MTT}, Theorem 3.7]\label{maximals} Assume that $q \geq m$. The following holds:
$$\hGamma(\mathbf{Q})=(\hGamma(\mathbf{Q})\cap \cC(\negQ))+\Theta(\negQ)$$
$$\hLambda(\mathbf{Q})=(\hLambda(\mathbf{Q})\cap \cC(\negQ))+\Theta(\negQ).$$
\end{theorem}

The concept of absolute maximality was used in \cite[Theorem 3.4]{MTT} to provide a description of $\hH(\negQ)$, whose motivation comes from \cite{Ma} for classical Weierstrass semigroups. Assuming that $q \geq m$, it was proved that
$$\hH(\mathbf{Q})=\{\lub(\negalpha^1,\ldots,\negalpha^m) \ : \ \negalpha^1,\ldots,\negalpha^m\in \hGamma(\mathbf{Q})\},$$
where
$$\mbox{lub}(\negalpha^1,\ldots,\negalpha^m):=(\max\{\alpha^1_1,\ldots,\alpha^m_1\},\ldots,\max\{\alpha^1_m,\ldots,\alpha^m_m\})\in \ZZ^m.$$
For $\negbeta\in \ZZ^m$, consider the following sets absolute maximal elements related with $\cL(D_\negalpha)$
$$\hGamma(\negbeta):=\{\negalpha\in \hGamma(\negQ) \ : \ \negalpha\leq \negbeta\},$$
where $\negalpha\leq \negbeta$ means $\alpha_i\leq \beta_i$ for all $i\in I$; see \cite[Thm. 3.5 and Cor. 3.6]{MTT} for details.

\subsection{Relative maximals and discrepancies} 
In \cite{DP}, Duursma and Park introduced the concept of discrepancies. As shown in \cite{TT}, it is possible to express notion of absolute maximal elements in generalized Weierstrass semigroups by using discrepancies, which has shown an important tool for explicit computations. In the rest of this section, we explore such similar connections with relative maximal elements.

\begin{definition}\label{def:discrepancy} 
Let $P$ and $Q$ be distinct rational points $P$ and $Q$ on $\mathcal{X}$. A divisor $A \in \mbox{Div}(\mathcal{X})$ is called a \textit{discrepancy} with respect to $P$ and $Q$ if $\mathcal{L}(A)\neq \mathcal{L}(A-Q)$ and $\mathcal{L}(A-P)=\mathcal{L}(A-P-Q)$.
\end{definition}

The next technical lemma will be an important tool in computations with discrepancies.


\begin{lemma}[\cite{fulton}, Noether's Reduction Lemma] \label{lemma noether}
Let $D\in \mbox{Div}(\cX)$, $P \in \mathcal{X}$, and let $K$ be a canonical divisor on $\cX$. If $\ell(D)>0$ and $\ell(K-D-P) \neq \ell(K-D)$, then $\ell(D+P)=\ell(D)$.
\end{lemma}

The following result establishes equivalences for the notion of relative maximality given in Definition \ref{def:maximals}. We remark that the equivalence $(1) \Leftrightarrow (5)$ is an adaptation of \cite[Lem. 1.3]{DM} for generalized Weierstrass semigroups at several points.

\begin{proposition} \label{relmax}
Let $\negalpha\in \ZZ^m$ and assume that $q \geq m$. The following statements are equivalent:
\begin{enumerate}[\rm (1)]
\item $\negalpha$ is relative maximal;
\item $\nabla(\negalpha)=\emptyset$ and $\ell(D_{\negalpha})=\ell(D_{\negalpha-\textbf{1}})+(m-1)$;
\item for $i,j\in I$ with $j\neq i$, $D_{\negalpha-\textbf{1}}+Q_i+Q_j$ is discrepancy w.r.t. $Q_i$ and $Q_j$;
\item there exists $i\in I$ such that $D_{\negalpha-\textbf{1}}+Q_i+Q_j$ is discrepancy w.r.t. $Q_i$ and $Q_j$, for any $j\neq i$;
\item there exists $i\in I$ such that $\nabla_i(\negalpha)=\emptyset$ and $\nabla_{i,j}(\negalpha)\neq \emptyset$ for every $j\neq i$.
\end{enumerate}
\end{proposition}
\begin{proof}
$(1) \Rightarrow (2)$: It is clear that $\nabla(\negalpha)=\emptyset$, since $\negalpha$ is relative maximal. By Proposition \ref{prop:elements}(2), we have that $\nabla_1(\negalpha)=\nabla_1^m(\negalpha-\textbf{1}+\textbf{e}_1)=\emptyset$ is equivalent to $\ell(D_{\negalpha-\textbf{1}} +Q_1)=\ell(D_{\negalpha-\textbf{1}})$. For $i=1,\ldots,m-1$, let $J_i=\{1,2,\ldots,i\}$. Hence, as $\nabla_{J_i}(\negalpha)\subseteq \nabla_i^m(\negalpha-\textbf{1}+\textbf{1}_{J_i})$ and $\nabla_{J_i}(\negalpha)\neq \emptyset$ for $i=2,\ldots,m-1$, we have $\ell(D_{\negalpha-\textbf{1}}+\sum_{j=1}^i Q_j)=\ell(D_{\negalpha-\textbf{1}}+\sum_{j=1}^{i-1} Q_j)+1$ again by Proposition \ref{prop:elements}(2). Since $\negalpha\in \hH(\negQ)$, we get $\ell(D_{\negalpha})=\ell(D_{\negalpha}-Q_m)+1$. Therefore, $\ell(D_{\negalpha})=\ell(D_{\negalpha-\textbf{1}})+(m-1)$.

$(2) \Rightarrow (3)$: Note that $\nabla(\negalpha)=\emptyset$ implies $\nabla_i(\negalpha)=\nabla_i^{m}(\negalpha-\textbf{1}+\negei)=\emptyset$ for any $i\in I$, and so, by Proposition \ref{prop:elements}(2), we obtain $\ell(D_{\negalpha-\textbf{1}}+Q_i)=\ell(D_{\negalpha-\textbf{1}})$ for $i\in I$. Since $\ell(D_{\negalpha})=\ell(D_{\negalpha-\textbf{1}})+(m-1)$ and $\ell(D_{\negalpha-\textbf{1}}+Q_i)=\ell(D_{\negalpha-\textbf{1}})$ for $i\in I$, we have $\ell(D_{\negalpha-\textbf{1}}+Q_i+Q_j)=\ell(D_{\negalpha-\textbf{1}}+Q_i)+1$ for any $j\neq i$, which proves (3).

$(3) \Rightarrow (4)$: Trivial.

$(4) \Rightarrow (5)$: By Proposition \ref{prop:elements}(2) we have that ${\nabla_i^m(\negalpha-\textbf{1}+\negei)=\nabla_i(\negalpha)= \emptyset}$, since $\ell(D_{\negalpha-\textbf{1}}+Q_i)=\ell(D_{\negalpha-\textbf{1}})$. For $j\neq i$, we have $\ell(D_{\negalpha-\textbf{1}}+Q_i+Q_j)\neq \ell(D_{\negalpha-\textbf{1}}+Q_i)$ and $\ell(D_{\negalpha-\textbf{1}}+Q_i+Q_j)\neq \ell(D_{\negalpha-\textbf{1}}+Q_j)$, and from Proposition \ref{prop:elements}(2), there exist $\negbeta\in\nabla_i^m(\negalpha-\textbf{1}+\negei+\negej)$ and $\negbeta'\in\nabla_j^m(\negalpha-\textbf{1}+\negei+\negej)$. Therefore, $\mbox{lub}(\negbeta,\negbeta')\in\nabla_{i,j}(\negalpha)$ and so $\nabla_{i,j}(\negalpha)\neq \emptyset$.

$(5) \Rightarrow (1)$: For any $j\in I\backslash\{i\}$, there exists $\negbeta^j\in \nabla_{i,j}(\negalpha)$, and thus ${\negalpha=\mbox{lub}(\{\negbeta^j \ : \ j\neq i\})}\in \hH(\negQ)$. Now, if $\neggamma\in \nabla_k(\negalpha)$ for some $k\neq i$ , since $\neggamma\neq \negbeta^k$ with $\gamma_k=\beta^k_k=\alpha_k$, it follows from Lemma \ref{teclem} that there exists $\negbeta\in \nabla_i(\negalpha)$, contradicting our assumption. Hence $\nabla_k(\negalpha)=\emptyset$ for every $k\in I$, and it only remains to verify that $\nabla_J(\negalpha)\neq \emptyset$ for $J\subsetneq I$ with $\#J\geq 2$.
For $k,l\in I\backslash\{i\}$, as $\nabla_{i,k}(\negalpha)\neq \emptyset$ and $\nabla_{i,l}(\negalpha)\neq \emptyset$, there exists $\negbeta'\in \nabla_{k,l}(\negalpha)$ by Lemma \ref{teclem}. Notice that for $J\subsetneq I$ with $\#J> 2$, we can write $J=J_1\cup\cdots \cup J_s$, for not necessarily disjoint subsets $J_t\subsetneq J$ with $\#J_t=2$.  Hence, by the previous remark, there exists $\neggamma^t\in \nabla_{J_t}(\negalpha)$, for $t=1,\ldots,s$. Therefore $\negbeta=\mbox{lub}(\{\neggamma^1,\ldots,\neggamma^s\})\in \nabla_J(\negalpha)$, which proves (1) and completes the proof.
\end{proof}

The next result is due to Delgado and appears in \cite{D} as an adaptation of \cite[Thm. 1.5]{DM} for the case of generalized Weierstrass semigroups. According to Delgado, we have the following relation between elements outside $\hH(\negQ)$ and relative maximal elements in $\hH(\negQ)$.

\begin{proposition}[\cite{D}, p. 629]\label{gaps}
Let $\negbeta\in \ZZ^m$ and assume that $q\geq m$. Then $\negbeta\not\in \hH(\negQ)$ if and only if $\negbeta\in \overline{\nabla}_i(\negbeta^*)$ for some $\negbeta^*\in
\hLambda(\negQ)$ and $i\in I$.
\end{proposition}

The proof of this result, omitted in \cite{D}, derives from a key result which is interesting enough to be stated separately. In the next result, we present its proof by adapting the notations and concepts presented in \cite[Thm. 1.5]{DM}. This will be important to what we will see later.

\begin{proposition}\label{lem:main}
Assume that $q \geq m$. Let $\negbeta=(\beta_1,\ldots,\beta_m)\in \ZZ^m$ and let $i\in I$. The following statements are equivalent:
\begin{enumerate}[\rm (1)]
\item  $\ell(D_{\negbeta}) = \ell(D_{\negbeta} - Q_i)$;
\item $\negbeta\in \overline{\nabla}_i(\negbeta^*)$ for some $\negbeta^*\in \hLambda (\negQ)$;
\item $\beta_i\neq \max\{\alpha_i \ : \ \negalpha=(\alpha_1,\ldots,\alpha_m)\in \hGamma(\negbeta)\}$.
\end{enumerate}
\end{proposition}
\begin{proof} For $\negbeta=(\beta_1,\ldots,\beta_m)\in \ZZ^m$ and $i,j\in I$, let us consider the following set
$$\nabla_i^j(\negbeta):=\{\neggamma\in \hH(\negQ) \ : \ \gamma_i=\beta_i, \gamma_k\leq \beta_k \ \mbox{for } 1\leq k\leq j, \mbox{and } \gamma_s<\beta_s \ \mbox{for } s>j\}.$$
(1)$\Leftrightarrow$(2) By the maximality of $\negbeta^*$ we have that (2) implies (1). Now, let us assume, without loss of generality, that  $\nabla_1^m(\negbeta)=\emptyset$. First, note that there exists an integer $\beta'_m$ such that ${\nabla_1^m(\beta_1,\ldots,\beta_{m-1},\beta'_m)\neq \emptyset}$ (by Remark~\ref{rmk:2g}, it is enough to take $\beta'_m>2g(\cX)-\beta_1-\cdots-\beta_{m-1}$). Set $$\beta^*_m=\min\{\gamma_m \ : \ \neggamma=(\gamma_1,\ldots,\gamma_m)\in\nabla_1^m(\beta_1,\ldots,\beta_{m-1},\beta'_m)\}.$$
Now, let $\neggamma^m\in\nabla_1^m(\beta_1,\ldots,\beta_{m-1},\beta'_m)$ be such that $\gamma^m_m=\beta^*_m$, and set $\negbeta^m=(\beta_1,\ldots,\beta_{m-1},\beta^*_m)\in~\ZZ^m$. Note that
\begin{itemize}
\item $\beta_m<\beta^*_m$, since otherwise we would have $\neggamma^m\in \nabla^m_1(\negbeta)=\emptyset$;
\item $\neggamma^m\in \nabla_{1,m}^{m-1}(\negbeta^m)$, as $\neggamma^m\in\nabla_1^m(\beta_1,\ldots,\beta_{m-1},\beta'_m)$ and $\gamma^m_m=\beta^*_m$;
\item $\nabla_1^{m-1}(\negbeta^m)=\emptyset$, because of the minimality of $\beta^*_m$.
\end{itemize}
Again, it follows from Remark \ref{rmk:2g} that there exists an integer $\beta'_{m-1}$ in such a way that ${\nabla_1^{m-1}(\beta_1,\ldots,\beta_{m-2},\beta'_{m-1},\beta^*_m)\neq \emptyset}$. Set $$\beta^*_{m-1}=\min\{\gamma_{m-1} \ : \ \neggamma=(\gamma_1,\ldots,\gamma_m)\in\nabla_1^{m-1}(\beta_1,\ldots,\beta_{m-2},\beta'_{m-1},\beta^*_m)\}.$$
Let $\neggamma^{m-1}\in\nabla_1^m(\beta_1,\ldots,\beta_{m-2},\beta'_{m-1},\beta^*_m)$ such that $\gamma^{m-1}_{m-1}=\beta^*_{m-1}$. Let also $\negbeta^{m-1}=(\beta_1,\ldots,\beta_{m-2},\beta^*_{m-1},\beta^*_m)\in \ZZ^m$. For the same reasons as before we have
\begin{itemize}
\item $\beta_{m-1}<\beta^*_{m-1}$ and $\neggamma^{m-1}\in \nabla_{1,m-1}^{m-2}(\negbeta^{m-1})$;
\item $\nabla_1^{m-2}(\negbeta^{m-1})=\emptyset$.
\end{itemize}
Repeated application of this argument enables us to obtain $\beta^*_2,\ldots,\beta^*_m\in \ZZ$ and $\neggamma^2,\ldots,\neggamma^m\in~\hH(\negQ)$ such that, writing $\negbeta^*=(\beta_1,\beta^*_2,\ldots,\beta^*_m)\in \ZZ^m$, satisfy:
\begin{itemize}
\item $\beta_{j}<\beta^*_{j}$ and $\neggamma^{j}\in \nabla_{1,j}(\negbeta^*)$ for $j=2,\ldots,m$;
\item $\nabla_1(\negbeta^*)=\emptyset$.
\end{itemize}
It follows from Proposition \ref{relmax}(5) that $\negbeta^*$ is a relative maximal element of $\hH(\negQ)$, and we thus have $\negbeta\in \overline{\nabla}_1(\negbeta^*)$.

(1)$\Leftrightarrow$(3) If $\beta_i= \alpha_i$ for some $\negalpha\in \hGamma(\negbeta)$, it follows from Proposition \ref{prop:elements}(2) that $\negalpha\in \nabla^m_i(\negbeta)$, contradicting $\ell(D_\negbeta)=\ell(D_\negbeta-Q_i)$. Conversely, if $\ell(D_\negbeta)\neq\ell(D_\negbeta-Q_i)$, then $\nabla^m_i(\negbeta)\neq \emptyset$ by Proposition \ref{prop:elements}, and thus there exists $\negalpha\in \nabla^m_i(\negbeta)\cap \hGamma(\negQ)$, which implies that $\negalpha\in \hGamma(\negbeta)$ with $\alpha_i=\beta_i$, contrary to the assumption.
\end{proof}

We note that Propositions \ref{prop:elements}(2) and \ref{lem:main} yield equivalences to the condition $\ell(D_\negalpha)=\ell(D_\negalpha-Q_i)$ for $i\in I$; in the sequence we explore the consequences of these relations in order to establish connections with gaps and pure gaps at several points.

\section{A characterization of Weierstrass gaps at several points}

We now turn back to the proposal of providing a characterization of the Weierstrass gaps at several points in terms of a finite set of special elements. For this, we shall consider a curve $\cX$ over $\FF_q$ and $\negQ=(Q_1,\ldots,Q_m)$ under the same assumptions as before.

In a sequence of papers \cite{K,H, HK}, Homma and Kim furnished a description of $H(Q_1,Q_2)$ and $G(Q_1,Q_2)$, as well as $G_0(Q_1,Q_2)$, through the following key observation: setting $\sigma(\alpha):=\min\{\beta \ : \ (\alpha,\beta)\in H(Q_1,Q_2)\}$ for $\alpha\in G(Q_1)$, we have that $\sigma(\alpha)\in G(Q_2)$. Moreover, they showed that the assignment $\sigma$ provides a bijection between the Weierstrass gap sets $G(Q_1)$ and $G(Q_2)$ so that $\sigma$ can be regarded as a permutation of $\{1,\ldots,g\}$. Writing $\ell_1< \cdots < \ell_g$ and $\ell'_1< \cdots < \ell'_g$ the gap sequences at $Q_1$ and $Q_2$, respectively, let us consider the fundamental set
\begin{equation}\label{genset2}
\Gamma(Q_1,Q_2):=\{(\ell_i,\ell'_{\sigma(i)}) \ : \ i=1,\ldots,g\}.
\end{equation}
The set $\Gamma(Q_1,Q_2)$ was used in \cite{K} to yield a way of describing the elements of $H(Q_1,Q_2)$, and this approach was generalized later by Matthews \cite{Ma} to the general case, the so-called minimal generating sets of classical Weierstrass semigroups.

Regarding the Weierstrass gaps at $(Q_1,Q_2)$, it follows from the properties of $\sigma$ that the sets of gaps and pure gaps at $(Q_1,Q_2)$ are related to $\Gamma(Q_1,Q_2)$ as follows:
$$G(Q_1,Q_2)=\bigcup_{i=1}^g\left(\{(\ell_i,\beta) \ : \ 0\leq \beta< \ell'_{\sigma(i)}\}\cup \{(\beta,\ell'_{\sigma(i)}) \ : \ 0\leq \beta< \ell_i\}\right)$$
and
$$G_0(Q_1,Q_2)=\{(\ell_i,\ell'_j) \ : (i,\sigma^{-1}(j))\in R(\sigma)\},$$
where $R(\sigma)=\{(i,j) \ : \ i<j \ \mbox{and } \sigma(i)>\sigma(j)\}$. Notice that in terms of the settled notation in the previous section, we identify the sets
$$\{(\ell_i,\beta) \ : \ 0\leq \beta< \ell'_{\sigma(i)}\}=\overline{\nabla}_1(\ell_i,\ell'_{\sigma(i)})\cap \NN^2_0$$
and
$$\{(\beta,\ell'_{\sigma(i)}) \ : \ 0\leq \beta< \ell_i\}=\overline{\nabla}_2(\ell_i,\ell'_{\sigma(i)})\cap \NN^2_0,$$
which lead us to rewrite the sets of gaps and pure gaps at $(Q_1,Q_2)$ as
\begin{equation}\label{eq:gaps2}
G(Q_1,Q_2)=\bigcup_{i=1}^g (\overline{\nabla}(\ell_i,\ell'_{\sigma(i)})\cap \NN^2_0)
\end{equation}
and 
\begin{equation}\label{eq:puregaps2}
G_0(Q_1,Q_2)=\bigcup_{(i,j)}\left(\overline{\nabla}_1(\ell_i,\ell'_{\sigma(i)})\cap \overline{\nabla}_2(\ell_j,\ell'_{\sigma(j)}) \right).
\end{equation}
Observe furthermore that $\Gamma(Q_1,Q_2)$ is related with the notion of maximality in Definition~\ref{def:maximals}. By the properties of $\sigma$, the elements $(\ell_i,\ell'_{\sigma(i)})$, for $i=1,\ldots,g$, are exactly the maximal elements of $\hH(Q_1,Q_2)$ contained in $H(Q_1,Q_2)$. Building on these observations together with Lemma \ref{lem:main}, we state characterizations for Weierstrass gaps and pure gaps which do not seem to have been noticed previously. For this, let us consider $\Lambda(\negQ):=\hLambda(\negQ)\cap \NN_0^m$.
In the following we extend the description \eqref{eq:gaps2} of Weierstrass gaps at pairs given in \cite{HK} to the general case in terms of relative maximal elements of $\hH(\negQ)$ as follows.

\begin{theorem}\label{gaps} Assume that $q\geq m$. Then
$$G(\negQ)=\bigcup_{\negbeta^*\in\Lambda(\negQ)} (\overline{\nabla}(\negbeta^*)\cap \NN_0^m).$$
\end{theorem}
\begin{proof} As $\negalpha \in \NN^m_0$ belongs to $G(\negQ)$ if and only if $\ell(D_\negalpha)=\ell(D_\negalpha-Q_j)$ for some $j\in I$, it is equivalent to $\negalpha\in \overline{\nabla}(\negbeta^*)=\bigcup_{j=1}^m \overline{\nabla}_j(\negbeta^*)$ for some $\negbeta^*\in \hLambda(\negQ)$ by Proposition \ref{lem:main}.
\end{proof}

Employing the same idea as before, we have an expanded characterization of pure gaps at several points, generalizing the account \eqref{eq:puregaps2} given in \cite{HK}.

\begin{theorem}\label{puregaps} Assume that $q\geq m$. Then
$$G_0(\negQ)=\bigcup_{\scriptsize{(\negbeta^1,\ldots,\negbeta^m)\in\Lambda(\negQ)^m}}\left(\bigcap_{i=1}^m \overline{\nabla}_i(\negbeta^i)\right).$$
\end{theorem}
\begin{proof}
By Proposition \ref{lem:main}, $\ell(D_\negalpha)=\ell(D_\negalpha-Q_j)$ for all $j\in I$, if and only if $\negalpha\in\bigcap_{i=1}^m \overline{\nabla}_i(\negbeta^i)$ for ${\negbeta^1,\ldots,\negbeta^m\in \hLambda(\negQ)}$.
\end{proof}

\begin{remark}\label{rmk:puregaps} Assuming that $q \geq m$, we notice that an intersection $\bigcap_{i=1}^m \overline{\nabla}_i(\negbeta^i)$, for ${\negbeta^1,\ldots,\negbeta^m\in \hLambda(\negQ)}$ as in Theorem \ref{puregaps}, is either empty or a singleton consisting of a pure gap, namely $(\beta^1_1,\ldots,\beta^m_m)$. Consequently $$\bigcap_{i=1}^m \overline{\nabla}_i(\negbeta^i)\neq \emptyset$$
if and only if, for each $i\in I$,
$$\beta^i_i<\beta^j_i \quad \mbox{for all} \quad j\in I\backslash\{i\}.$$
It thus yields a procedure to determine such intersections, and therefore the whole set of classical pure gaps. Observe furthermore that
\begin{equation*}\label{glb}
G_0(\negQ)\subseteq\{\mbox{glb}(\negbeta^1,\ldots,\negbeta^m) \ : \ \negbeta^1,\ldots,\negbeta^m \in \Lambda(\negQ)\},
\end{equation*}
where
$$\mbox{glb}(\negbeta^1,\ldots,\negbeta^m):=(\min\{\beta^1_1,\ldots,\beta^m_1\},\ldots,\min\{\beta^1_m,\ldots,\beta^m_m\})\in \NN_0^m.$$
\end{remark}





\section{Gaps and pure gaps in certain curves with separated variables}
Let $\cX_{f,g}$ be a curve over $\Fq$ admitting a plane model of type
$$f(y)=g(x),$$
where $f(T), g(T)\in \mathbb{F}_q[T]$ with $\deg(f(T))=a $ and $\deg(g(T))=b$ satisfying $\mbox{gcd}(a,b)=1$. Suppose that $\cX_{f,g}$ has genus $(a-1)(b-1)/2$ and is geometrically irreducible. Suppose moreover that there exist $a+1$ distinct rational points $P_1, P_2, \ldots, P_{a+1}$ on $\cX_{f,g}$ such that
\begin{equation}\label{eq1}
a P_1 \sim P_2 + \cdots + P_{a+1}
\end{equation}
and
\begin{equation}\label{eq2}
\ b P_1 \sim b P_{j} \ \mbox{ for } j\in \{2,\ldots,a+1\},
\end{equation}
where $b$ is the smallest positive integer satisfying (\ref{eq2}) and ``$\sim$" represents the linear equivalence of divisors. Notice that, by the above assumptions, $H(P_1)=\langle a,b\rangle$.\\

Considering the region $\cC(\negP_m)$, the assumption \eqref{eq2} yields
$$\cC(\negP_m)=\{\negbeta\in \ZZ^m \ : \ 0\leq \beta_i<b \ \mbox{for } i=2,\ldots,m\}.$$

\begin{theorem} [Theorem 4.2, \cite{TT}] \label{absolute maximals} Let $P_1,\ldots,P_{a+1}$ be rational points on $\cX_{f,g}$ and $\cC(\negP_m)$ be as above. For $2\leq m\leq  a+1$, let
$$\negalpha^{i,m}=(a(b-i)-b(m-1),i,\ldots,i)\in \ZZ^m.$$
Then  $$\hGamma(\mathbf{P}_m)\cap \cC(\negP_m)=\{\negalpha^{i,m} \ : \ i=1,\ldots,b-1\}\cup \{{\bf 0}\}.$$
\end{theorem}

The previous result gave us a way to find the elements in the generalized Weierstrass semigroup $\hH(\negP_m)$ and consequently the Riemann-Roch spaces of divisors supported in $\{P_1,\ldots,P_{a+1}\}$; see \cite{TT} for details. In the next result, we will explicit the relative maximal elements of $\hH(\negP_m)$  in the region $\cC(\negP_m)$. According to Theorem \ref{maximals}, these elements determine all relative maximal elements of $\hH(\negP_m)$ and in particular the gaps and pure gaps at $\negP_m$.


\begin{theorem}\label{relativemaximals}
Let $P_1,\ldots,P_{a+1}$ be rational points on $\cX_{f,g}$ and $\cC(\negP_m)$ be as above. For ${2\leq m\leq a+1}$, let
$$\begin{array}{rcll}
\negbeta^{0,m} & = & \left(b(m-2),0,\ldots,0\right), & \mbox{and} \\
\negbeta^{i,m} & = & \left(a(b-i)-b,i,\ldots,i\right) & \mbox{for }i=1,\ldots,b-1.
\end{array}
$$
Then
$$\hLambda(\negP_m)\cap \cC(\negP_m)=\{\negbeta^{i,m} \ : \ i=0,1,\ldots,b-1\}.$$
\end{theorem}
\begin{proof} For $2\leq m\leq a+1$, let $\negbeta^{0,m} =  \left(b(m-2),0,\ldots,0\right)$ and $\negbeta^{i,m} = \left(a(b-i)-b,i,\ldots,i\right)$ for $i=1,\ldots,b-1$.
If $m=2$, then $\hGamma(\negP_m)=\hLambda(\negP_m)$ and so the result follows from the Theorem \ref{absolute maximals}. Now, let $m\geq 3$ and let us denote $R_m=\{\negbeta^{i,m} \ : \ i=0,1,\ldots,b-1\}$. First, we will prove that $R_m\subseteq \hLambda(\negP_m)\cap \cC(\negP_m)$. It is clear that $R_m\subseteq \cC(\negP_m)$. By Proposition~\ref{relmax}, it is sufficient to prove that $A=D_{\negbeta^{i,m}-\textbf{1}+\textbf{e}_1+\negej}$, with $0\leq i<b$, is discrepancy with respect to $P_1$ and $P_j$ for any $j\in I\backslash\{1\}$. In fact, from \eqref{eq1} and \eqref{eq2} there exist functions $h,g_2,\ldots,g_{a+1}\in \Fq(\cX_{f,g})$ such that
\begin{equation} \label{div}
\divv(h)=\sum_{k=2}^{a+1} P_k-aP_1 \ \quad \mbox{and} \ \quad \divv(g_j)=bP_j-bP_1, \ \mbox{ for }j=2,\ldots, a+1.
\end{equation}
 For $i\neq 0$, since
$$A=(a(b-i)-b)P_1+iP_j+(i-1)\sum_{k=2 \atop k\neq j}^m P_k,$$
and
$$\divv(h^{b-i}/g_j)=-(a(b-i)-b)P_1-iP_j+(b-i)\sum_{k=2 \atop k\neq j}^{a+1} P_k,$$
we have
$$h^{b-i}/g_j\in \cL(A)\backslash \cL(A-P_j).$$
Now, we must prove that $\cL(A-P_1)=\cL(A-P_1-P_j)$. By Lemma \ref{lemma noether}, it suffices to prove that $\cL(K-A+P_1+P_j)\neq \cL(K-A+P_j)$, where $K$ is a canonical divisor on $\cX_{f,g}$. Taking the canonical divisor $K=(ab-a-b-1)P_1$, we obtain
$$K-A+P_1+P_j=a(i-1)P_1-(i-1)\sum_{k=2}^m P_k.$$
As
$$h^{i-1}\in \cL(K-A+P_1+P_j)\backslash \cL(K-A+P_j),$$
we thus have that $A$ is a discrepancy with respect to $P_1$ and $P_j$ for $j\in I\backslash\{1\}$. Therefore, by Proposition \ref{relmax}(3), we get $\negbeta^{i,m}\in \hLambda(\negP_m)$ for $i=1,\ldots,b-1$. Hence, to conclude that $R_m\subseteq \hLambda(\negP_m)\cap \cC(\negP_m)$, it remains to verify that $\negbeta^{0,m}\in \hLambda(\negP_m)$.
Again, we will prove that $\cL(A)\neq \cL(A-P_1)$ and $\cL(K-A+P_1+P_j)\neq \cL(K-A+P_j)$, where $A=D_{\negbeta^{0,m}-\textbf{1}+\textbf{e}_1+\textbf{e}_j}$ for all $j\in I\backslash\{1\}$. Notice that
$$A=b(m-2)P_1-\sum_{k=2 \atop k\neq j}^{m}P_k$$
and we obtain that
$$g_2\cdots g_{j-1}\cdot g_{j+1}\cdots g_m \in \cL(A)\backslash \cL(A-P_1).$$
Moreover, since $K-A+P_1+P_j=(a(b-1)-b(m-1))P_1+\sum_{k=2}^m P_k$, we get
$$\frac{h^{b-1}}{g_2\cdots g_m}\in \cL(K-A+P_1+P_j)\backslash \cL(K-A+P_j)$$
and so we conclude that $R_m\subseteq \hLambda(\negP_m)\cap \cC(\negP_m)$.
Now, let $\negbeta\in \hLambda(\negP_m)\cap \cC(\negP_m)$. Since $\nabla_j^m(\negbeta)\neq \emptyset$ for any $j\in I$, there exists an absolute maximal element $\negalpha^{i,m}\in \nabla_2^m(\negbeta)$, and thus $\beta_2=i$ and $\beta_3\geq i$. Similarly, there exists an absolute maximal element $\negalpha^{i',m}\in \nabla_3^m(\negbeta)$, and thus $\beta_3=i'$ and $\beta_2\geq i'$. Hence $i=i'$. Proceeding in the same way with pairs of the remaining indexes, we conclude that there exists an absolute maximal element $\negalpha^{i,m}\in \bigcap_{j=2}^m \nabla_j^m(\negbeta)$ and in particular that $\beta_j=i$ for $j=2,\ldots,m$. As $\negbeta\in\hLambda(\negP_m)$ and $m\geq 3$, it follows that $\negbeta\neq \negalpha^{i,m}$ and thus $\beta_1>\alpha^{i,m}_1$. Hence, for each ${\negbeta\in~\hLambda(\negP_m)\cap \cC(\negP_m)}$, there exists a unique $\negalpha^{i,m}\in~\hGamma(\negP_m)\cap \cC(\negP_m)$ such that $\negalpha^{i,m}\in~\nabla_{I\backslash\{1\}}(\negbeta)$. Therefore, $\#(\hGamma(\negP_m)\cap \cC(\negP_m))\geq \#(\hLambda(\negP_m)\cap \cC(\negP_m))$. As $\#R_m=\#(\hGamma(\negP_m)\cap \cC(\negP_m))$ and ${R_m\subseteq \hLambda(\negP_m)\cap \cC(\negP_m)}$, we have $ \hLambda(\negP_m)\cap \cC(\negP_m)=R_m$, which proves the result.
\end{proof}

We can thus obtain all relative maximals elements of $\hH(\negP_m)$ as follows.

\begin{corollary}
Let $P_1,\ldots,P_{a+1}$ be rational points on $\cX_{f,g}$ and $\cC_m$ be as above. For ${2\leq m\leq a+1}$, we have
$$\hLambda(\negP_m)=\left\{\left(a(b-i)-b-b\textstyle\sum_{j=2}^m d_j,i+bd_2,\ldots,i+bd_m\right) \ : \ {d_j\in \ZZ \ \mbox{for } j=2,\ldots,m \atop \mbox{and } \ 1\leq i\leq b-1}\right\}$$
$$\bigcup \left\{\left(b(m-2)-b\textstyle\sum_{j=2}^m d_j,bd_2,\ldots,bd_m\right) \ : \   d_j\in \ZZ \ \mbox{for } j=2,\ldots,m\right\}.$$
\end{corollary}
\begin{proof}
It follows from Theorem \ref{maximals} that $\hLambda(\negP_m)=(\hLambda(\negP_m)\cap \cC(\negP_m))+\Theta(\negP_m)$. By \eqref{eq2}, we have that $\Theta(\negP_m)$ is generated by $m$-tuples of type $$\negeta^i=(0,\ldots,0,-b,\underbrace{b}_{ith},0,\ldots,0),$$ for $i=2,\ldots,m$. Hence an element of $\Theta(\negP_m)$ has the form $(-b\sum_{i=2}^m d_i,bd_2,\ldots,bd_m)$ for $d_2,\ldots,d_m\in \ZZ$, and the result thus follows from Theorem \ref{relativemaximals}.
\end{proof}

\begin{corollary} \label{cor:relmax}
Let $P_1,\ldots,P_{a+1}$ be rational points on $\cX_{f,g}$ and $\cC_m$ as above. For ${2\leq m\leq a+1}$, we have
$$\Lambda(\negP_m)=\left\{\left(a(b-i)-b-b\textstyle\sum_{j=2}^m d_j,i+bd_2,\ldots,i+bd_m\right) \ : \ {{{d_j\in \NN_0 \ \text{for } j=2,\ldots,m \ \text{with}} \atop \ { a(b-i)-b(1+\sum_{j=2}^m d_j)\geq 0}} \atop {\text{and } \ 1\leq i\leq b-1}}^{}\right\}.$$
\end{corollary}
\begin{proof}
As $\Lambda(\negP_m)=\hLambda(\negP_m)\cap \NN_0^m$, it is a direct consequence of the previous result.
\end{proof}

The following result presents explicitly the gaps of the curves $\cX_{f,g}$ at the $m$-tuples $\negP_m$. 

\begin{proposition} \label{gapsgf} Let $P_1,\ldots,P_{a+1}$ be rational points on $\cX_{f,g}$ as above and ${2\leq m\leq a+1}$. For $i=1,\ldots,b-1$, let
$$S_{i1}=\left\{\left(a(b-i)-b(1+j),\beta_2,\ldots,\beta_m\right)\in \NN_0^m \ : \ {j=0,\ldots,\lfloor\textstyle\frac{a(b-i)-b}{b} \rfloor, \ d_t\in \NN_0, \atop \ \beta_t<i+bd_t, \ \mbox{and } \sum_{i=2}^m d_t=j }\right\},$$
and for each $2\leq k\leq m$, let
$$S_{ik}=\left\{(\beta_1,\ldots,\beta_{k-1},i+bj,\beta_{k+1},\ldots,\beta_m)\in \NN_0^m \ : {j=0,\ldots,\lfloor\textstyle\frac{a(b-i)-b}{b} \rfloor, \ d_t\in \NN_0, \ \beta_t<i+bd_t,\atop \ \beta_1<a(b-i)-b(1+d_1), \ \mbox{and }d_1-\sum_{t=2 \atop t\neq k}^m d_t=j} \right\}.$$ Then
$$G(\negP_m)=\bigcup_{i=1}^{b-1}\left( \bigcup_{k=1}^m S_{ik}\right).$$
\end{proposition}
\begin{proof} It follows from Theorem \ref{gaps} by noticing that
$$S_{ik}=\overline{\nabla}_k\left(a(b-i)-b-b\textstyle\sum_{j=2}^m d_j,i+bd_2,\ldots,i+bd_m\right)\cap \NN_0^m.$$
\end{proof}

Let $$A^*:=\left\{a(b-i)-b(1+j) \ : \ i=1,\ldots,b-1, \ j=0,\ldots,\lfloor\textstyle
\frac{a(b-i)-b}{b}\rfloor\right\}$$
and
$$A:=\left\{i+bj \ : \ i=1,\ldots,b-1, \ j=0,\ldots,\lfloor\textstyle\frac{a(b-i)-b}{b}\rfloor\right\}.$$
We have
$$G_0(\negP_m)\subseteq  A^*\times A^{m-1}.$$

\begin{example} Let $\ell$ be a prime power and let $r$ be a positive integer. Let $\cX$ be the Norm-Trace curve defined over $\FF_{\ell^{r}}$ by the affine equation
$$x^{\frac{\ell^r-1}{\ell-1}}=y^{\ell^{r-1}}+y^{\ell^{r-2}}+\cdots+y$$
of genus $(\ell^r-1)(\frac{\ell^r-1}{\ell-1}-1)/2$. Let $P_\infty$ be the point at infinity $(0:1:0)$ of $\cX$ and let $P_{0b_j}$ be the points $(0:b_j:1)$ with $b_j\in \FF_{\ell^{r}}$ such that $b_j^{\ell^{r-1}}+b_j^{\ell^{r-2}}+\cdots+b_j=0$ for $j=1,\ldots,\ell^{r-1}$. 
In this curve, we have $a=\ell^{r-1}$ and $b=(\ell^r-1)/(\ell-1)$ because
$$\divv(x)=\sum_{j=1}^{\ell^{r-1}} P_{0b_j}-\ell^{r-1} P_\infty \quad \mbox{and} \quad \divv(y-b_j)=\textstyle\frac{\ell^r-1}{\ell-1}(P_{0b_j}-P_\infty) \ \mbox{for } j=1,\ldots,\ell^{r-1}.$$ Hence, according to Theorem \ref{maximals} and Theorem \ref{relativemaximals}, the relative maximal elements in the generalized Weierstrass semigroup $\hH(P_\infty,P_{0b_1},\ldots,P_{0b_{m-1}})$, for $2\leq m\leq \ell^{r-1}+1$, are completely determined by $(\ell^r-1)/(\ell-1)+m-1$ elements. For instance, for $\ell=4$ and $r=2$, we have the Hermitian curve of genus $6$. According to Corollary \ref{cor:relmax}, the relative maximal elements of $\hH(P_\infty,P_{0b_1},P_{0b_{2}})$ in $\NN^3$ are
\begin{equation}\label{eqpuregaps}
\begin{array}{c}
\neggamma^1=(1, 1, 11), \ \neggamma^2=(1, 6, 6), \ \neggamma^3=(1, 11, 1), \ \neggamma^4=(2, 2, 7), \ \neggamma^5=(2, 7, 2),  \\
\neggamma^6=(3, 3, 3), \ \neggamma^7=(6, 1, 6), \ \neggamma^8=(6, 6, 1), \ \neggamma^9=(7, 2, 2), \mbox{and } \neggamma^{10}=(11, 1, 1).
\end{array}
\end{equation}
By Theorem \ref{puregaps} and Remark \ref{rmk:puregaps}, the pure gaps at $(P_\infty,P_{0b_1},P_{0b_{2}})$ are
$$\begin{array}{c}
(1, 1, 1), \ (2, 1, 1), \ (1, 1, 2), \ (1, 2, 1), \ (3, 1, 1), \ (1, 1, 3), \ (1, 3, 1), \ (2, 1, 2), \\
(2, 2, 1), \ (1, 2, 2), \ (3, 1, 2), \ (2, 1, 3), \ (3, 2, 1), \ (1, 2, 3), \ (2, 3, 1), \ \mbox{and }(1, 3, 2).
\end{array}$$
They are determined from the following intersections:
\begin{small}
$$(1, 1, 1)\in\overline{\nabla}_1(\neggamma^2)\cap \overline{\nabla}_2(\neggamma^7)\cap \overline{\nabla}_3(\neggamma^8)$$
$$(2, 1, 1)\in\overline{\nabla}_1(\neggamma^4)\cap \overline{\nabla}_2(\neggamma^7)\cap \overline{\nabla}_3(\neggamma^8)=\overline{\nabla}_1(\neggamma^5)\cap \overline{\nabla}_2(\neggamma^7)\cap \overline{\nabla}_3(\neggamma^8)$$
$$(1, 1, 2)\in\overline{\nabla}_1(\neggamma^2)\cap \overline{\nabla}_2(\neggamma^7)\cap \overline{\nabla}_3(\neggamma^5)=\overline{\nabla}_1(\neggamma^2)\cap \overline{\nabla}_2(\neggamma^7)\cap \overline{\nabla}_3(\neggamma^9)$$
$$(1, 2, 1)\in\overline{\nabla}_1(\neggamma^2)\cap \overline{\nabla}_2(\neggamma^4)\cap \overline{\nabla}_3(\neggamma^8)=\overline{\nabla}_1(\neggamma^2)\cap \overline{\nabla}_2(\neggamma^9)\cap \overline{\nabla}_3(\neggamma^8)$$
$$ (3, 1, 1)\in\overline{\nabla}_1(\neggamma^6)\cap \overline{\nabla}_2(\neggamma^7)\cap \overline{\nabla}_3(\neggamma^8) $$
$$(1, 1, 3)\in\overline{\nabla}_1(\neggamma^2)\cap \overline{\nabla}_2(\neggamma^7)\cap \overline{\nabla}_3(\neggamma^6) $$
$$ (1, 3, 1)\in\overline{\nabla}_1(\neggamma^2)\cap \overline{\nabla}_2(\neggamma^6)\cap \overline{\nabla}_3(\neggamma^8) $$
$$(2, 1, 2)\in\overline{\nabla}_1(\neggamma^4)\cap \overline{\nabla}_2(\neggamma^7)\cap \overline{\nabla}_3(\neggamma^9)$$
$$(2, 2, 1)\in\overline{\nabla}_1(\neggamma^5)\cap \overline{\nabla}_2(\neggamma^9)\cap \overline{\nabla}_3(\neggamma^8)$$
$$(1, 2, 2)\in\overline{\nabla}_1(\neggamma^2)\cap \overline{\nabla}_2(\neggamma^4)\cap \overline{\nabla}_3(\neggamma^5)$$
$$(3, 1, 2)\in\overline{\nabla}_1(\neggamma^6)\cap \overline{\nabla}_2(\neggamma^7)\cap \overline{\nabla}_3(\neggamma^9)$$
$$(2, 1, 3)\in\overline{\nabla}_1(\neggamma^4)\cap \overline{\nabla}_2(\neggamma^7)\cap \overline{\nabla}_3(\neggamma^6)$$
$$(3, 2, 1)\in\overline{\nabla}_1(\neggamma^6)\cap \overline{\nabla}_2(\neggamma^9)\cap \overline{\nabla}_3(\neggamma^8)$$
$$(1, 2, 3)\in\overline{\nabla}_1(\neggamma^2)\cap \overline{\nabla}_2(\neggamma^4)\cap \overline{\nabla}_3(\neggamma^6)$$
$$(2, 3, 1)\in\overline{\nabla}_1(\neggamma^5)\cap \overline{\nabla}_2(\neggamma^6)\cap \overline{\nabla}_3(\neggamma^8)$$
$$(1, 3, 2)\in\overline{\nabla}_1(\neggamma^2)\cap \overline{\nabla}_2(\neggamma^6)\cap \overline{\nabla}_3(\neggamma^5)$$
\end{small}
The list of all gaps at $(P_\infty,P_{0b_1},P_{0b_{2}})$ can be obtained from Proposition \ref{gapsgf} by using the elements in \eqref{eqpuregaps} as
$$
\overline{\nabla}(\neggamma^1)\cap \NN_0^3= \{(1, 0, s) \ : 0\leq s\leq 10\}\cup \{(0, 1, s ) \ : 0\leq s\leq 10\}\cup \{(0, 0, 11)\}
$$
$$
\overline{\nabla}(\neggamma^2)\cap \NN_0^3=\{(1,r,s) \ : \ 1\leq r,s\leq 5\}\cup \{(0,6,s) \ : \ 0\leq s\leq 5\}\cup \{(0,r,6) \ : \ 0\leq s\leq 5\}
$$
$$
\overline{\nabla}(\neggamma^3)\cap \NN_0^3=\{(1, r, 0) \ : \ 0\leq r\leq 10 \}\cup \{(0,11,0)\} \cup\{(0, r , 1) \ : \ 0\leq r\leq 10 \}
$$
$$
\overline{\nabla}(\neggamma^4)\cap \NN_0^3=\{(2,r,s), (r,2,s)  \ : \ 0\leq r\leq 1 \ \mbox{and } 0\leq s\leq 6\}\cup \{(r,s,7) \ : \ 0\leq r,s\leq 1\}
$$
$$
\overline{\nabla}(\neggamma^5)\cap \NN_0^3=\{(2,r,s), (s,r,2)  \ : \ 0\leq s\leq 1 \ \mbox{and } 0\leq r\leq 6\}\cup \{(r,7,s) \ : \ 0\leq r,s\leq 1\} 
$$
$$
\overline{\nabla}(\neggamma^6)\cap \NN_0^3=\{(3,r,s) \ : \ 0\leq r, s\leq 2\}\cup \{(r,3,s) \ : \ 0\leq r, s\leq 2\}\cup \{(r,s,3) \ : \ 0\leq r, s\leq 2\}
$$
$$
\overline{\nabla}(\neggamma^7)\cap \NN_0^3=\{(6,0,s) \ : \ 0\leq  s\leq 5\}\cup \{(r,1,s) \ : \ 0\leq  r,s\leq 5\}\cup \{(r,0,6) \ : \ 0\leq  r\leq 5\}
$$
$$
\overline{\nabla}(\neggamma^8)\cap \NN_0^3=\{(6,s,0) \ : \ 0\leq  s\leq 5\}\cup \{(r,6,0) \ : \ 0\leq  r\leq 5\}\cup \{(r,s,1) \ : \ 0\leq  r,s\leq 5\}
$$
$$
\overline{\nabla}(\neggamma^9)\cap \NN_0^3=\{(7,r,s) \ : \ 0\leq r, s\leq 1\}\cup \{(r,2,s), (r,s,2) \ : \ 0\leq r\leq 6 \ \mbox{and } 0\leq s\leq 1\}	
$$
$$
\overline{\nabla}(\neggamma^{10})\cap \NN_0^3= \{(11, 0, 0)\}\cup \{(s, 1, 0) \ : 0\leq s\leq 10\}\cup \{(s, 0, 1 ) \ : 0\leq s\leq 10\}
$$

Taking now $\ell=2$ and $r=3$, $\cX$ is a curve of genus $9$. In this case, the relative maximal elements of $\hH(P_\infty,P_{0b_1},P_{0b_{2}})$ in $\NN^3$ are
$$(17, 1, 1), \ (10, 1, 8), \ (3, 1, 15), \ (13, 2, 2), \ (6, 2, 9), \ (9, 3, 3), \ (2, 3, 10), $$ $$(5, 4, 4), \ 	(1, 5, 5), \ (10, 8, 1), \ (3, 8, 8), \ (6, 9, 2), \ (2, 10, 3), \mbox{and } (3, 15, 1).$$
In the same way as before, we may use these elements to list all gaps by considering their associated sets and to find that the pure gaps are
\begin{footnotesize}
$$(1, 1, 1), (1, 1, 2), (1, 1, 3), (1, 1, 4), (1, 2, 1), (1, 2, 2), (1, 2, 3), (1, 2, 4), (1, 3, 1), (1, 3, 2), (1, 3, 3), (1, 3, 4), $$
$$ (1, 4, 1), (1, 4, 2), (1, 4, 3), ( 2, 1, 1), (2, 1, 2), (2, 1, 3), (2, 1, 4), (2, 1, 8), (2, 1, 9), (2, 2, 1), (2, 2, 2), (2, 2, 3),   $$
$$  (2, 2, 4), (2, 2, 8), (2, 3, 1), (2, 3, 2), (2, 4, 1), (2, 4, 2), (2, 8, 1), (2, 8, 2), (2, 9, 1), (3, 1, 1), (3, 1, 2), (3, 1, 3), $$
$$  (3, 1, 4), (3, 2, 1), (3, 2, 2), (3, 2, 3), (3, 2, 4), (3, 3, 1), (3, 3, 2), (3, 4, 1), (3, 4, 2), (5, 1, 1), (5, 1, 2), (5, 1, 3),  $$ 
$$ (5, 2, 1), (5, 2, 2), (5, 2, 3), (5, 3, 1), (5, 3, 2), (6, 1, 1), (6, 1, 2), (6, 1, 3), (6, 2, 1), (6, 3, 1), (9, 1, 1), (9, 1, 2), (9, 2, 1).$$
\end{footnotesize}

\end{example}


\end{document}